\documentclass[11pt]{amsart}

\usepackage{amscd,amssymb,amsmath,graphicx,verbatim,xypic}
\usepackage{wasysym}
\usepackage[dvips]{hyperref}
\usepackage[TS1,OT1,T1]{fontenc}
\usepackage{lscape} %landscape
\usepackage{fullpage} %full page
\usepackage{pslatex} %times new roman
\usepackage{tikz}
\usetikzlibrary{shapes,arrows}

\newtheorem{theorem}{Theorem}[section]
\newtheorem{lemma}[theorem]{Lemma}
\newtheorem{corollary}[theorem]{Corollary}
\newtheorem{proposition}[theorem]{Proposition}

\theoremstyle{definition}
\newtheorem{definition}[theorem]{Definition}

\newtheorem{example}[theorem]{Example}

\theoremstyle{remark}
\newtheorem{remark}[theorem]{Remark}

\newcommand{\Dcal}{\ensuremath{\mathcal{D}}}
\newcommand{\Xcal}{\ensuremath{\mathcal{X}}}
\newcommand{\Ycal}{\ensuremath{\mathcal{Y}}}

\newcommand{\Lbb}{\ensuremath{\mathbb{L}}}

\newcommand{\Ucal}{\ensuremath{\mathcal{U}}}
\newcommand{\Kcal}{\ensuremath{\mathcal{K}}}

\newcommand{\Kbb}{\mathbb{K}}

\newcommand{\ra}{\rightarrow}

\numberwithin{equation}{section}

\begin{document}
\title{From ring epimorphisms to universal localisations}
\author{Frederik Marks, Jorge Vit{\'o}ria}
\address{Frederik Marks, Jorge Vit{\'o}ria, Institute of algebra and number theory, University of Stuttgart, Pfaffenwaldring 57, D-70569 Stuttgart, Germany}
\email{marks@mathematik.uni-stuttgart.de, vitoria@mathematik.uni-stuttgart.de}
\thanks{The first named author is supported by DFG-SPP 1489 and the second named author by DFG-SPP 1388.}

\begin{abstract}
For a fixed ring, different classes of ring epimorphisms and localisation maps are compared. In fact, we provide sufficient conditions for a ring epimorphism to be a universal localisation.  Furthermore, we consider recollements induced by some homological ring epimorphisms and investigate whether they yield recollements of derived module categories.

\end{abstract}
\maketitle

\section{Introduction}
It is well-known that Ore localisations yield ring epimorphisms with a flatness condition. 
Different generalisations of Ore localisation, notably localisation with respect to a Gabriel filter and universal localisation, usually lack this flatness property.
Localisations with respect to Gabriel filters generalise Ore localisations from a torsion-theoretic point of view. From a homological perspective, however, these localisation maps are often difficult to deal with. In fact, they are not always ring epimorphisms. Still, this setting is large enough to include all flat ring epimorphisms and these localisations are called perfect (see \cite{St} for details).
Universal localisations, as developed by Cohn (\cite{Cohn}) and Schofield (\cite{Sch}), provide a technique that largely differs from the one above. In particular, they yield ring epimorphisms satisfying some nice homological properties. Universal localisations have shown to be useful in algebraic K-theory  (\cite{NR1}, \cite{NR2}) and the study of tilting modules and derived module categories in representation theory (\cite{LA}, \cite{ALK1}, \cite{CX}, \cite{CX2}, \cite{CX3}).

In \cite{LA}, both universal and perfect localisations were used to construct (large) tilting modules. Furthermore, \cite{LA} compares perfect and universal localisations for semihereditary rings and Pr\"ufer domains. Also, in \cite{KS}, it was proved that universal localisations are in bijection with homological ring epimorphisms for hereditary rings. These results motivate the study of universal localisations from a homological point of view, which we further in this paper, namely through our first theorem.
\\

\textbf{Theorem A} (Theorem \ref{Main}) \textit{Let $f:A\ra B$ be a ring epimorphism such that $B$ is a finitely presented left $A$-module of projective dimension less or equal than one. Then $f$ is homological if and only if it is a universal localisation.}\\

Recent work uses universal localisations to construct interesting examples of recollements of derived module categories (\cite{ALK1}, \cite{CX}, \cite{CX2}, \cite{CX3}). In this setting, we prove the following theorem. \\

\textbf{Theorem B} (Theorem \ref{Main 2}) \textit{Let $f:A\rightarrow B$ be a homological ring epimorphism such that $B$ is a finitely presented left $A$-module of projective dimension less or equal than one. If $Hom_A(coker(f), ker(f))=0$ holds then the derived restriction functor $f_*$ induces a recollement of derived module categories 
\begin{equation}\nonumber
\begin{xymatrix}{\mathcal{D}(B)\ar[r]^{}&\mathcal{D}(A)\ar@<1.5ex>[l]_{}\ar@<-1.5ex>[l]_{}\ar[r]^{}&
\mathcal{D}(End_{\Dcal(A)}(K_f)),\ar@<1.5ex>_{}[l]\ar@<-1.5ex>_{}[l]}
\end{xymatrix}
\end{equation}
where $K_f$ is the cone of $f$ in $\Dcal(A)$. Moreover, if $B$ is a finitely presented projective left $A$-module then there is an isomorphism of rings $End_{\Dcal(A)}(K_f)\cong A/\tau_B(A)$, where $\tau_B(A)$ is the trace of $B$ in $A$. 
}\\

Note that theorem A cannot hold in full generality since universal localisations are not always homological ring epimorphisms (see example \ref{ex1}) just as homological ring epimorphisms are not necessarily universal localisations, notably through Keller's example in \cite{Ke2}. Using different methods, theorem A has also been proved independently by Chen and Xi in \cite{CX3} (corollary 3.7).

Theorem B yields recollements in which both outer terms are derived module categories. These recollements are particularly relevant to recent results obtained in \cite{ALK2}, \cite{ALK3} and \cite{LY}, where a Jordan-H{\"o}lder-type theorem for derived module categories of some rings has been proved. Such a property cannot, however, hold for all rings and a counter example can be constructed using universal localisations (\cite{CX}).

This paper is organised as follows. In section 2 we recall some preliminaries and prove some easy facts. Remark \ref{reflection} and lemma \ref{finite bij}, in particular, give information on how to construct examples in our setting. Section 3 contains theorem A and consequences for the cases of finite, injective and surjective ring epimorphisms. Also, following subsection 2.4, we generalise the comparison between universal localisations, localisations with respect to Gabriel filters and flat ring epimorphisms initiated in \cite{LA}. In section 4 we prove theorem B, while examples illustrating this result are given in section 5. In particular, we use our methods to obtain a large class of algebras which are not derived simple.

\section{Ring epimorphisms and localisations}
Throughout, $A$ will be a ring with unit and $\Kbb$ a field. We will denote the category of left (respectively, right) $A$-modules by $A\mbox{-}Mod$ (respectively, $Mod\mbox{-}A$), its subcategory of finitely generated modules by $A\mbox{-}mod$ (respectively, $mod\mbox{-}A$) and its subcategory of finitely generated projective modules by $A\mbox{-}proj$. The derived category of left $A$-modules will be denoted by $\Dcal(A)$.
\subsection{Ring epimorphisms}
We will be discussing some types of ring epimorphisms. Recall that a ring epimorphism is just an epimorphism in the category of rings with unit. Two ring epimorphisms $f:A\rightarrow B$ and $g:A\rightarrow C$ are said to be equivalent if there is a ring isomorphism $h: B\rightarrow C$ such that $g=hf$. We then say that $B$ and $C$ lie in the same epiclass of $A$.

\begin{proposition}[\cite{St}, Proposition XI.1.2]\label{ring epi}
For a ring homomorphism $f:A\rightarrow B$, the following statements are equivalent.
\begin{enumerate}
\item $f$ is a ring epimorphism;
\item The restriction functor $f_*: B\mbox{-}Mod\rightarrow A\mbox{-}Mod$ (respectively, $f_\#: Mod\mbox{-}B\rightarrow Mod\mbox{-}A$) is fully faithful;
\item $f\otimes_A B=B\otimes_A f: B\rightarrow B\otimes_AB$ is an isomorphism of $B\mbox{-}B$-bimodules;
\item $B\otimes_A coker(f)=0$.
\end{enumerate}
Moreover, the functor $B\otimes_A-$ (respectively, $-\otimes_A B$) is left adjoint to $f_*$ (respectively, $f_\#$).
\end{proposition}

Consider the following sequence of left $A$-modules given by a ring epimorphism $f:A\ra B$
\[ \xymatrix{0\ar[r]& ker(f)\ar[r]& A\ar[r]^f & B\ar[r]& coker(f)\ar[r]& 0,}\]
which we unfold into two short exact sequences, namely
\begin{equation}\label{seq 1 ker}
\xymatrix{0\ar[r]& ker(f)\ar[r]& A\ar[r]^{\bar{f}\ \ }& f(A)\ar[r]&0,}
\end{equation}
\begin{equation}\label{seq 2 coker}
\xymatrix{0\ar[r]& f(A)\ar[r]& B\ar[r]& coker(f)\ar[r]&0.}
\end{equation}

The following easy observations follow from proposition \ref{ring epi}. 

\begin{corollary}\label{ker and coker}
Let $f:A\ra B$ be a ring epimorphism. The following assertions hold.
\begin{enumerate}
\item $B\otimes_A f(A)\cong B\otimes_A B \cong B$;
\item $B\otimes_A ker(f)\cong Tor_1^A(B,f(A))$;
\item If $Tor_1^A(B,B)=0$ then $Tor_1^A(B,coker(f))=0$.
\end{enumerate}
\end{corollary}
\begin{proof}
To prove (1), consider the commutative diagram given by the epi-mono factorisation of $f$
\[\xymatrix{A\ar[rr]^f\ar[rd]_{\bar{f}} && B\\ & f(A)\ar[ru]}\]
and apply to it the functor $B\otimes_A -$. By proposition \ref{ring epi}, $B\otimes_A f:B\otimes_A A\ra B\otimes_A B$ is an isomorphism and, therefore, the induced epimorphism $B\otimes_A \bar{f}$ is also a monomorphism.

The statements (2) and (3) follow from (1) by considering the long exact sequences given by applying the functor $B\otimes_A-$ to the sequences (\ref{seq 1 ker}) and (\ref{seq 2 coker}), respectively.

\end{proof}

Epiclasses of a ring $A$ can be classified by suitable subcategories of $A\mbox{-}Mod$. For a ring epimorphism $f:A\ra B$ we denote by $\mathcal{X}_B$ the essential image of the restriction functor.

\begin{theorem}[\cite{GdP}, Theorem 1.2, \cite{GeLe}, \cite{I}, Theorem 1.6.1]\label{bireflective}
There is a bijection between:
\begin{enumerate}
\item ring epimorphisms $A\ra B$ up to equivalence;
\item bireflective subcategories $\Xcal_B$ of $A\mbox{-}Mod$ (respectively, $Mod\mbox{-}A$), i.e., strict full subcategories of $A\mbox{-}Mod$ (respectively, $Mod\mbox{-}A$) closed under products, coproducts, kernels and cokernels.
\end{enumerate}
If $A$ is a finite dimensional $\Kbb$-algebra, this bijection can be restricted between:
\begin{enumerate}
\item ring epimorphisms $A\ra B$ up to equivalence, where $B$ is a finite dimensional $\Kbb$-algebra;
\item bireflective subcategories $\Xcal_B$ of $A\mbox{-}mod$ (respectively, $mod\mbox{-}A$), i.e., strict full functorially finite subcategories of $A\mbox{-}mod$ (respectively, $mod\mbox{-}A$) closed under kernels and cokernels.
\end{enumerate}
\end{theorem}

Given a ring epimorphism $f:A\ra B$ consider the adjunction in proposition \ref{ring epi}. For a left $A$-module $M$, let $\psi_M: M\rightarrow B\otimes_A M$ be the unit of this adjunction at $M$. Clearly, we have that
\[\psi_M(m)=1_B\otimes m, \ \forall m\in M.\]
Note that $\psi_M$ for a left $B$-module $N$ is an isomorphism.
The following easy lemma shows that the map $\psi_M$ is the $\Xcal_B$-reflection of the left $A$-module $M$.

\begin{lemma}\label{loc mod}
Let $f:A\ra B$ be a ring epimorphism and $M$ a left $A$-module. For any left $B$-module $N$ and for any $A$-homomorphism $g:M\ra N$, $g$ factors uniquely through $\psi_M$.
\end{lemma}
\begin{proof}
Since the map $\psi_N$ is an isomorphism, we can define a homomorphism of $A$-modules $$\tilde{g}:=\psi_N^{-1}\circ (B\otimes_A g).$$ It is clear that $g=\tilde{g}\circ\psi_M$ and, by construction, $\tilde{g}$ is the unique map satisfying this property.
\end{proof}

\begin{remark}\label{reflection}
In particular, note that the ring epimorphism $f:A\ra B$, regarded as a homomorphism of $A$-modules, is a $\Xcal_B$-reflection. Moreover, if $A$ is a finite dimensional $\Kbb$-algebra, then $f$ can be seen as the sum of the reflections of the indecomposable projective $A$-modules.
\end{remark}

\subsection{Flat and finite ring epimorphisms}
\begin{definition}
A ring epimorphism $f:A\rightarrow B$ is said to be 
\begin{itemize}
\item flat, if $f$ turns $B$ into a flat left $A$-module;
\item finite, if $f$ turns $B$ into a finitely generated projective left $A$-module;
\item 1-finite, if $f$ turns $B$ into a finitely presented left $A$-module of projective dimension less or equal than one.
\end{itemize}
\end{definition}

Clearly, every finite ring epimorphism is flat and 1-finite. Conversely, the following result holds.

\begin{proposition}[\cite{Cu}, Corollary 1.4]\label{cun}
If $A$ is a perfect ring, then a ring epimorphism $f:A\ra B$ is flat if and only if it is finite.
\end{proposition}

\begin{remark}\label{A-proj}
For a perfect ring $A$, a ring epimorphism $A\rightarrow B$ is finite if and only if every finitely generated projective left $B$-module is finitely generated and projective as a left $A$-module. Equivalently, $B$ is finitely generated as a left $A$-module and for all $M$ in $B\mbox{-}mod$ its projective cover in $A\mbox{-}mod$ is also a left $B$-module.
\end{remark}

Recall that finite dimensional $\Kbb$-algebras are perfect rings. From remark \ref{A-proj} and theorem \ref{bireflective} we get the following immediate lemma.

\begin{lemma}\label{finite bij} Let $A$ be a finite dimensional $\Kbb$-algebra. There is a bijection between
\begin{enumerate}
\item finite ring epimorphisms $A\ra B$ up to equivalence;
\item bireflective subcategories $\Xcal_B$ of $A\mbox{-}mod$ (respectively, $mod\mbox{-}A$) such that projective objects of $\Xcal_B$ are projective $A$-modules.

\end{enumerate}
\end{lemma}

\subsection{Homological ring epimorphisms}

We are interested in ring epimorphisms with particularly nice homological properties. Following Geigle and Lenzing (\cite{GeLe}), a ring epimorphism $f:A\ra B$ is said to be homological if $Tor_i^A(B,B)=0$, for all $i>0$. 

For any ring epimorphism $f:A\ra B$, we denote by $K_f$ the object 
\[\xymatrix{A\ar[r]^f&B}\]
in the category of complexes of left $A$-modules, where $A$ lies in position $-1$. Note that, regarded as an object of $\Dcal(A)$, $K_f$ is isomorphic to the cone of the map $f$, seen as a map of complexes concentrated in degree zero. The following well-known result is an analogue of proposition \ref{ring epi} for homological ring epimorphisms.

\begin{proposition}\label{hom ring epi}
The following are equivalent for a ring homomorphism $f:A\rightarrow B$.  
\begin{enumerate}
\item $f$ is a homological ring epimorphism;
\item The derived restriction functor $f_*:\Dcal(B)\rightarrow \Dcal(A)$ is fully faithful;
\item $B\otimes_A^\Lbb f: B\rightarrow B\otimes_A^\Lbb B$ is an isomorphism in $\Dcal(A)$;
\item $B\otimes_A^\Lbb K_f=0$.
\end{enumerate}
Moreover, the functor $B\otimes_A^\Lbb-$ is left adjoint to $f_*$.
\end{proposition}
\begin{proof}
The fact that (1) is equivalent to (2) can be found in \cite{GeLe} (theorem 4.4).  

It is easy to see that (1) is equivalent to (3). Indeed, note that $H^0(B\otimes_A^\Lbb f)=B\otimes_A f$ is an isomorphism if and only if $f$ is a ring epimorphism. Also, for $i>0$, $H^i(B\otimes_A^\Lbb f)=Tor_i^A(B,f)$ is the zero map and it is an isomorphism if and only if $H^i(B\otimes_A^{\Lbb} B)=Tor_i^A(B,B)=0$. 

Finally, we check that (3) is equivalent to (4). Consider the triangle in $\Dcal(A)$
\[\xymatrix{A\ar[r]^f& B\ar[r]& K_f\ar[r]& A[1]}\]
and apply to it the triangle functor $B\otimes_A^\Lbb-$. Clearly, $B\otimes_A^\Lbb f$ is an isomorphism if and only if $B\otimes_A^\Lbb K_f=0$, thus finishing the proof.
\end{proof}

Homological ring epimorphisms of $A$ play a role in understanding how to \textit{decompose} the derived category $\Dcal(A)$ into other triangulated categories. This \textit{decomposition} is formalised by the notion of recollement.

\begin{definition}
Let $\Xcal, \Ycal, \Dcal$ be triangulated categories. A
recollement of $\Dcal$ by $\Xcal$ and $\Ycal$ is a diagram of six triangle
functors, satisfying the properties below.
\begin{equation}\nonumber
\begin{xymatrix}{\mathcal{Y}\ar[r]^{i_*}&\mathcal{D}\ar@<3ex>[l]_{i^!}\ar@<-3ex>[l]_{i^*}\ar[r]^{j^*}&\mathcal{\mathcal{X}}\ar@<3ex>_{j_*}[l]\ar@<-3ex>_{j_!}[l]}.
\end{xymatrix}
\end{equation}\newpage
\begin{enumerate}
\item $(i^\ast,i_\ast)$,\,$(i_*,i^!)$,\,$(j_!,j^*)$ ,\,$(j^\ast,j_\ast)$
are adjoint pairs;

\item $i_\ast,\,j_\ast,\,j_!$  are full embeddings;

\item  $i^!\circ j_\ast=0$ (and thus also $j^*\circ i_*=0$ and
$i^\ast\circ j_!=0$);

\item for each $Z\in \Dcal$ there are triangles \[i_* i^!Z\to
Z\to j_\ast j^\ast Z\to i_* i^!Z[1]\]
\[j_! j^* Z\to Z\to
i_\ast i^\ast Z\to j_!j^*Z[1].\]
\end{enumerate}
\end{definition}

We now recall the following result from \cite{NS}, stating how homological ring epimorphisms give rise to recollements.

\begin{theorem}[\cite{NS}, \S4]\label{rec hom}
Let $f:A\ra B$ be a homological ring epimorphism. Then the derived restriction functor $f_*$ induces a recollement
$$\begin{xymatrix}{\mathcal{D}(B)\ar[r]^{f_*}&\mathcal{D}(A)\ar@<2.5ex>[l]\ar@<-2.5ex>[l]\ar[r]&Tria(K_f)\ar@<2.5ex>[l]\ar@<-2.5ex>[l]},
\end{xymatrix}$$
where $Tria(K_f)$ denotes the smallest triangulated subcategory of $\Dcal(A)$ containing $K_f$ and closed under coproducts.

\end{theorem}

\subsection{Universal localisations}

The following theorem defines and shows existence of universal localisations.

\begin{theorem}[\cite{Sch}, Theorem 4.1]\label{uni loc}
Let $A$ be a ring and $\Sigma$ a set of maps between finitely generated projective left $A$-modules. Then there is a ring $A_\Sigma$, unique up to isomorphism, and a ring homomorphism $f_\Sigma: A\rightarrow A_\Sigma$ such that
\begin{enumerate}
\item $A_\Sigma \otimes_A \sigma$ is an isomorphism of left $A$-modules for all $\sigma\in\Sigma$;
\item every ring homomorphism $g:A\rightarrow B$ such that $B\otimes_A \sigma$ is an isomorphism for all $\sigma\in\Sigma$ factors in a unique way through $f_\Sigma$, i.e., there is a commutative diagram of the form
\begin{equation}\nonumber
\xymatrix{A\ar[rr]^g\ar[rd]_{f_\Sigma}&&B\\ & A_\Sigma.\ar[ru]_{\exists! \tilde{g}}}
\end{equation}
\end{enumerate}
\end{theorem}

We say that the ring $A_\Sigma$ in the theorem is the universal localisation of $A$ at $\Sigma$. It is well-known that the homomorphism $f_\Sigma: A\rightarrow A_\Sigma$ is a ring epimorphism with $Tor_1^A(A_\Sigma,A_\Sigma)=0$ (\cite{Sch}). 
The functor $A_\Sigma\otimes_A -$ is called the localisation functor of the universal localisation and it is left adjoint to the restriction functor $f_{\Sigma *}: A_{\Sigma}\mbox{-}Mod\ra A\mbox{-}Mod$ (see proposition \ref{ring epi}). For a left $A$-module M we call the $\Xcal_{A_{\Sigma}}$-reflection $\psi_M$ the localisation map of $M$ (see lemma \ref{loc mod}).

We can also define universal localisations with respect to a certain set of $A$-modules. Indeed, let $\Ucal$ be a set of finitely presented left $A$-modules of projective dimension less or equal than one. We denote by $A_{\Ucal}$ the universal localistaion of $A$ at $\Sigma=\{\sigma_U|\,U\in\Ucal\}$, where $\sigma_U:P\rightarrow Q$ is a projective resolution of $U$ in $A\mbox{-}mod$. Note that $A_{\Ucal}$ is well-defined by \cite{Cohn} and we will call it the universal localisation of A at $\Ucal$.
The following easy example shows that universal localisations do not, in general, yield homological ring epimorphisms.

\begin{example}\label{ex1}
Let $A$ be the quotient of the path algebra over $\mathbb{K}$ of the quiver $$\xymatrix{1\ar[r]^{\alpha} & 2\ar[r]^{\beta} & 3}$$ by the ideal generated by $\beta\alpha$. Consider the universal localisation of A at $\Ucal:=\{P_2\}$. Note that $A_{\Ucal}$ and $A/Ae_2A$ lie in the same epiclass of $A$. It is easy to check that $Tor^A_2(A_{\Ucal},A_{\Ucal})\not= 0$ and, hence, the ring epimorphism $A\rightarrow A_{\Ucal}$ is not homological.
\end{example}

\subsection{Localisations with respect to Gabriel filters}

These localisations generalise the torsion-theoretical properties of Ore localisations. In fact, right Gabriel filters in a ring $A$ are in bijection with hereditary torsion classes in $A\mbox{-}Mod$. Also, in contrast with Ore or universal localisation, the localisation functor associated to a Gabriel filter is not necessarily the tensor product with the localised ring. For details and definitions we refer the reader to \cite{St}. In what follows we discuss some properties of these localisations that motivate some of the questions answered in this paper. We start by discussing how flat ring epimorphisms relate to this notion of localisation.

\begin{theorem}[\cite{St}, Theorem XI.2.1, Proposition XI.3.4]\label{Gabriel}
A localisation with respect to a Gabriel filter yields a flat ring epimorphism if and only if the localisation functor is naturally equivalent to the tensor product with the localised ring.
Moreover, any flat ring epimorphism $f:A\ra B$ lies in the same epiclass as the localisation of $A$ with respect to a Gabriel filter of right ideals of $A$. 
\end{theorem}

A localisation with respect to a Gabriel filter is said to be perfect if it yields a flat ring epimorphism. The following corollary establishes a first connection between universal localisations, localisations with respect to Gabriel filters and flat ring epimorphisms.

\begin{corollary}\label{uni+Gab}
If a universal localisation is a localisation with respect to a Gabriel filter then it is perfect, i.e., it yields a flat ring epimorphism.
\end{corollary}
\begin{proof}
The localisation functor of a universal localisation is the tensor product with the localised ring. The result then follows from theorem \ref{Gabriel}.
\end{proof}

\section{A sufficient condition for universal localisation}

In this section we provide sufficient conditions on a ring epimorphism for it to be a universal localisation. Recall that a quasi-isomorphism is a morphism of complexes inducing isomorphisms in the cohomologies. 

\begin{proposition}\label{resolution}
Let $f:A\ra B$ be a ring epimorphism. The following are equivalent.
\begin{enumerate}
\item There is a quasi-isomorphism from $P_f$, a complex $\small{\xymatrix{P_f^{-1}\ar[r]^g& P_f^0}}$ of projective left $A$-modules, to $K_f$;
\item $B$ is a left $A$-module of projective dimension less or equal than one.
\end{enumerate}
Moreover, if these conditions hold, $B$ is finitely presented if and only if $P_f$ can be chosen as a complex of finitely generated projective left $A$-modules.

\end{proposition}
\begin{proof}
(1) $\Rightarrow$ (2) Suppose we have a quasi-isomorphism as in the diagram
\begin{equation}\label{two rows}
\xymatrix{0\ar[r]& ker(g)\ar[d]^k_\cong\ar[r]^{k_1}&P_f^{-1}\ar[d]^{\pi_2}\ar[r]^{g}&P_f^0\ar[d]^{\pi_1}\ar[r]^{c_1\ \ \ \ \ }& coker(g)\ar[d]^c_\cong\ar[r]&0\\ 0\ar[r]& ker(f)\ar[r]^{k_2}& A\ar[r]^f&B \ar[r]^{c_2\ \ \ \ \ }& coker(f)\ar[r]&0.}
\end{equation}
Define a complex as follows:
\[\xymatrix{0\ar[r]&P_f^{-1}\ar[r]^{p_1\ \ }&A \oplus  P_f^0\ar[r]^{\ \ \ p_2}& B\ar[r]&0,}\]
\[\begin{array}{cc}p_1:P_f^{-1}\longrightarrow A\oplus P_f^0 & p_2:A\oplus P_f^0\longrightarrow B\\ x\mapsto (\pi_2(x), g(x)) &\ \ \ \ \ \ \ \ \ (y,z)\mapsto f(y)-\pi_1(z).\end{array}\]
It is easy to check, by diagram chasing in (\ref{two rows}), that this is a short exact sequence. Hence, $B$ has projective dimension less or equal than one.

(2) $\Rightarrow$ (1): Choose a projective resolution of $B$ of shortest length
\[\xymatrix{0\ar[r]&P_1^B\ar[r]^h&P_0^B\ar[r]^\pi&B\ar[r]&0}\]
and consider a Cartan-Eilenberg resolution of $K_f$ given by
\[ \xymatrix{& P_1^B\ar[d]^{h} \\ A\ar[r]^{\hat{f}}\ar@{.>}[d]_{id}&P_0^B\ar@{.>}[d]^\pi\\ A\ar[r]^f& B}\]
It is well-known (see \cite{W}, \S5.7) that there is a quasi-isomorphism from its total complex
\[\xymatrix{A\oplus P_1^B\ar[rr]^{\hat{f}+h}&& P_0^B}\]
to $K_f$, thus finishing the proof.
\end{proof}

\begin{remark}
This proposition can be easily generalised to $B$ of any finite projective dimension. Since our focus is on 1-finite ring epimorphisms, it is convenient to keep the statement and proof as above.
\end{remark}

The following theorem shows that certain homological ring epimorphisms can be characterised as universal localisations.

\begin{theorem}\label{Main} 
Let $f: A\ra B$ be a 1-finite ring epimorphism. Then $f$ is homological if and only if it is a universal localisation.
\end{theorem}
\begin{proof}
Suppose that $f$ is a universal localisation. Then $Tor_1^A(B,B)=0$ and, since $B$ is a left $A$-module of projective dimension less or equal than one, $f$ is homological.

Conversely, let $P_f$ be a complex $\xymatrix{P^{-1}_f\ar[r]^g&P^0_f}$ of finitely generated projective left $A$-modules quasi-isomorphic to $K_f$, which exists by proposition \ref{resolution}. Since $f$ is homological, by proposition \ref{hom ring epi}, we have
$$0=B\otimes_A^\Lbb K_f\cong B\otimes_A^\Lbb P_f=B\otimes_A P_f$$ in $\Dcal(A)$, showing that $B\otimes_A g$ is an isomorphism of left $A$-modules. Therefore, by theorem \ref{uni loc}, there is a commutative diagram of ring epimorphisms
\[\xymatrix{A\ar[rr]^f\ar[rd]_{f_g}&&B\\ & A_{\{g\}}\ar[ru]_{h}}\]
showing that, in particular, the essential images of the corresponding restriction functors for right modules satisfy, by proposition \ref{ring epi},
$$\Xcal_B\subseteq  \Xcal_{A_{\{g\}}} \subseteq Mod\mbox{-}A.$$
In order to prove the reverse inclusion, we will see that $A_{\{g\}}\otimes_Af$ is an isomorphism of left (and right) $A$-modules.
To do so, consider the short exact sequence 
\begin{equation}\label{seq 1 ker g}
\xymatrix{0\ar[r]& ker(g)\ar[r]& P_f^{-1}\ar[r]^{\bar{g}\ \ }& g(P_f^{-1})\ar[r]&0}
\end{equation}
induced by the map $g$. Observe that a similar argument to the one in the proof of corollary \ref{ker and coker}(1) shows that $A_{\{g\}}\otimes_A \bar{g}$ is an isomorphism. 
Using the commutative diagram (\ref{two rows}) given by the quasi-isomorphism from $P_f$ to $K_f$ and applying the functor $A_{\{g\}}\otimes_A-$ to the short exact sequences (\ref{seq 1 ker g}) and (\ref{seq 1 ker}) we get the following diagram of left $A$-modules
\[\xymatrix{A_{\{g\}}\otimes_Aker(g)\ar[r]^0\ar[d]^{A_{\{g\}}\otimes_A k}&A_{\{g\}}\otimes_A P_f^{-1}\ar[r]^{\cong\ \ \ }\ar[d]& A_{\{g\}}\otimes_Ag(P_f^{-1})\ar[r]\ar[d]&0\\ A_{\{g\}}\otimes_Aker(f)\ar[r]^{\ \ \ \  A_{\{g\}}\otimes_A k_1}&A_{\{g\}}\otimes_A A\ar[r]^{A_{\{g\}}\otimes_A\bar{f}\ \ }&A_{\{g\}}\otimes_Af(A)\ar[r]&0.}\]
It shows that, since $A_{\{g\}}\otimes_A k$ is an isomorphism, $A_{\{g\}}\otimes_A k_1=0$ and thus $A_{\{g\}}\otimes_A \bar{f}$ is an isomorphism. Now, applying the functor $A_{\{g\}}\otimes_A-$ to the sequence (\ref{seq 2 coker}), we get 
\[\xymatrix{Tor_1^A(A_{\{g\}}, coker(f))\ar[r]& A_{\{g\}}\otimes_A f(A)\ar[r]& A_{\{g\}}\otimes_A B\ar[r]&0.}\]
In order to compute $Tor_1^A(A_{\{g\}}, coker(f))$, consider a projective resolution of $coker(f)$ of the form
\[\xymatrix{...\ar[r] & P^{-2}\ar[r]^d&P_f^{-1}\ar[r]^g&P_f^0\ar[r]& coker(f)\ar[r]&0}\]
and apply to it the functor $A_{\{g\}}\otimes_A-$. By definition, $A_{\{g\}}\otimes_A g$ is an isomorphism and, therefore, the first cohomology of the new complex is zero. This shows precisely that $Tor_1^A(A_{\{g\}}, coker(f))=0$ and, thus, using the epi-mono factorisation of $f$, we can conclude that  
\[A_{\{g\}}\otimes_A f:A_{\{g\}}\otimes_A A\ra A_{\{g\}}\otimes_A B\] is an isomorphism of left $A$-modules. It is, however, easy to check that this is also an isomorphism of right $A$-modules. Hence, $A_{\{g\}}$ has a natural right $B$-module structure, i.e, it lies in $\Xcal_B$. Since $A_{\{g\}}$ is a generator of $\Xcal_{A_{\{g\}}}$, this shows that $\Xcal_{A_{\{g\}}}\subseteq \Xcal_B$ and, thus, $\Xcal_{A_{\{g\}}}=\Xcal_B$. By proposition \ref{ring epi}, this means that $A_{\{g\}}$ and $B$ lie in the same epiclass of $A$ and, therefore, are isomorphic.
\end{proof}

\begin{remark}
As mentioned in the introduction, theorem \ref{Main} can be derived from independent current work of Chen and Xi by observing that, under our assumptions, the generalised localisation in \cite{CX3} (corollary 3.7) is a universal localisation.
\end{remark}

\begin{remark}
Note that, for a homological 1-finite ring epimorphism $f:A\rightarrow B$, the above proof together with the proof of proposition \ref{resolution} explicitly constructs a map g in $A\mbox{-}proj$ with $B\cong A_{\{g\}}$. Indeed, $g$ depends only on the choice of a projective resolution of B of shortest length in $A\mbox{-}mod$.
\end{remark}

In particular, for finite ring epimorphisms, we have the following result. 

\begin{corollary}\label{cor finite}
Let $f:A\rightarrow B$ be a finite ring epimorphism. Then $B$ lies in the same epiclass of $A$ as the universal localisation $A_{\left\{f\right\}}$, where $f$ is seen as an element of $A\mbox{-}proj$.
\end{corollary}

With further assumptions on the ring epimorphism $f$, the universal localisation in theorem \ref{Main} takes a particularly nice form.

\begin{corollary}\label{inj surj}
Let $f:A\rightarrow B$ be a homological 1-finite ring epimorphism. The following holds.
\begin{enumerate}
\item If $f$ is injective then $coker(f)=B/A$ is a finitely presented $A$-module of projective dimension less or equal than one and $B$ and $A_{\{B/A\}}$ lie in the same epiclass of $A$. 
\item If $f$ is surjective then $ker(f)$ is a finitely presented projective $A$-module and $B$ and $A_{\{ker(f)\}}$ lie in the same epiclass of $A$.
\end{enumerate}
Moreover, if $A$ is a finite dimensional $\mathbb{K}$-algebra and $f$ is surjective then $B$ and $A/AeA$ lie in the same epiclass of $A$, for some idempotent $e$ in $A$.
\end{corollary}

\begin{proof}
Let $P_f$ be a complex $\xymatrix{P^{-1}_f\ar[r]^g&P^0_f}$ of finitely generated projective left $A$-modules quasi-isomorphic to $K_f$, which exists by proposition \ref{resolution}.
\begin{enumerate}
\item Since $f$ is injective, $g$ is injective and $coker(f)\cong coker(g)$ is a finitely presented A-module of projective dimension less or equal than one. By theorem \ref{Main}, it follows that $B$ lies in the same epiclass of $A$ as $A_{\{g\}}=A_{\{coker(f)\}}$.

\item Since $f$ is surjective, $g$ is surjective and thus a split map. It follows that $ker(f)\cong ker(g)$ is a finitely presented projective A-module. Again, by theorem \ref{Main}, we get that $B$ lies in the same epiclass of $A$ as $A_{\{g\}}$, which is easily seen to be the universal localisation $A_{\{0\rightarrow ker(f)\}}=A_{\{ker(f)\}}$.
\end{enumerate}

Note that, if $f$ is surjective then $ker(f)$ is an idempotent ideal of $A$, since we have $$0=Tor^A_1(B,B)=Tor_A^1(A/ker(f),A/ker(f))=ker(f)/ker(f)^2.$$ Thus, if $A$ is a finite dimensional $\Kbb$-algebra then $ker(f)$ is generated by an idempotent $e$ in $A$.
\end{proof}

As a consequence of theorem \ref{Main} we can also establish a comparison between universal localisations and localisations with respect to Gabriel filters, motivated by the results in \cite{LA}.

\begin{corollary}
Let $A$ be a perfect ring and $f:A\ra B$ a ring epimorphism. Then $f$ is a universal localisation and a localisation with respect to a Gabriel filter if and only if $f$ is flat.
\end{corollary}
\begin{proof}
If $f$ is both a universal localisation and a localisation with respect to a Gabriel filter, it is flat by corollary \ref{uni+Gab}. Conversely, if $f$ is flat, it is a localisation with respect to a Gabriel filter by theorem \ref{Gabriel}. By proposition \ref{cun}, since $A$ is perfect, $f$ is finite and, thus, a universal localisation by corollary \ref{cor finite}.
\end{proof}

\section{Recollements of derived module categories}

We will now use homological 1-finite ring epimorphisms to construct recollements of derived module categories.
For two left $A$-modules $M,N$ we denote by $\tau_M(N)$ the trace of $M$ in $N$, i.e., the submodule of $N$ given by the sum of the images of all $A$-homomorphisms from $M$ to $N$.
\begin{theorem}\label{Main 2}
Let $f:A\rightarrow B$ be a homological 1-finite ring epimorphism with 
$Hom_A(coker(f), ker(f))=0.$ Then the derived restriction functor $f_*$ induces a recollement of derived module categories 
\begin{equation}\nonumber
\begin{xymatrix}{\mathcal{D}(B)\ar[r]^{}&\mathcal{D}(A)\ar@<1.5ex>[l]_{}\ar@<-1.5ex>[l]_{}\ar[r]^{}&
\mathcal{D}(End_{\Dcal(A)}(K_f)).\ar@<1.5ex>_{}[l]\ar@<-1.5ex>_{}[l]}
\end{xymatrix}
\end{equation}
Moreover, if $f$ is finite then there is an isomorphism of rings $End_{\Dcal(A)}(K_f)\cong A/\tau_B(A)$.
\end{theorem}

\begin{proof}
By theorem \ref{rec hom}, we have the following recollement of triangulated categories induced by the derived restriction functor $f_*$
$$\xymatrix{\mathcal{D}(B)\ar[r]^{}&\mathcal{D}(A)\ar@<1.5ex>[l]_{}\ar@<-1.5ex>[l]_{}\ar[r]^{}&
Tria(K_f).\ar@<1.5ex>_{}[l]\ar@<-1.5ex>_{}[l]}$$
Since $B$ is 1-finite, by proposition \ref{resolution}, $K_f$ is quasi-isomorphic to $P_f$, a complex $\xymatrix{P_f^{-1}\ar[r]^g&P_f^0}$ of finitely generated projective left $A$-modules, and therefore it is compact in $\Dcal(A)$. We will prove that it is exceptional. Recall that (see, for example, \cite{W}, corollary 10.4.7), for all $X$ in $\Dcal(A)$, $Hom_{\Dcal(A)}(K_f,X)\cong Hom_{\Kcal(A)}(P_f,X),$ where $\Kcal(A)$ denotes the homotopy category of complexes of left $A$-modules.
Clearly, for all $i\ge 2$ and $i\le -2$, we have $$Hom_{\Dcal(A)}(K_f,K_f[i])\cong Hom_{\mathcal{K}(A)}(P_f,K_f[i])=0.$$ 
Since, by assumption, we know that $$Hom_A(coker(g),ker(f))\cong Hom_A(coker(f),ker(f))=0,$$ we also get $$Hom_{\mathcal{D}(A)}(K_f,K_f[-1])\cong Hom_{\Kcal(A)}(P_f,K_f[-1])=0.$$ It remains to show that $$Hom_{\Dcal(A)}(K_f,K_f[1])\cong Hom_{\Kcal(A)}(P_f,K_f[1])=0.$$ Note that every element $\Phi$ in $Hom_{\Kcal(A)}(P_f,K_f[1])$ is uniquely determined by a morphism $\phi$ in $Hom_A(P_f^{-1},B)$ which, by lemma  \ref{loc mod}, factors through the $\Xcal_B$-reflection $\psi_{P_f^{-1}}$. This shows that $\Phi$ factors through $B\otimes_A P_f$, which is zero in $\Dcal(A)$ (see argument in the proof of theorem \ref{Main}). Since $B\otimes_A P_f$ is a two term complex, it is also zero in $\Kcal(A)$. Thus, we have $\Phi=0$ and $$Hom_{\Dcal(A)}(K_f,K_f[i])=0,\ \forall i\neq 0.$$

We conclude that $K_f$ is a compact exceptional object in $\mathcal{D}(A)$. Therefore, by a result of Keller (\cite{Ke1}, theorem 8.5), we get a recollement of derived module categories
$$\xymatrix{\mathcal{D}(B)\ar[r]^{}&\mathcal{D}(A)\ar@<1.5ex>[l]_{}\ar@<-1.5ex>[l]_{}\ar[r]^{}&
D(End_{\mathcal{D}(A)}(K_f)).\ar@<1.5ex>_{}[l]\ar@<-1.5ex>_{}[l]}$$

Suppose now that $f$ is finite and $P_f=K_f$. We will describe $End_{\mathcal{D}(A)}(K_f)\cong End_{\mathcal{K}(A)}(K_f)$. Note that, for any element $a$ in $A$, there is a unique morphism in $End_{\mathcal{K}(A)}(K_f)$ defined by $k_A(1_A)=a$ and $k_B(1_B)=f(a)$ as in the following commutative diagram
$$\xymatrix{...\ar[r] & 0\ar[r]\ar[d] & A\ar[r]^f\ar[d]_{k_A} & B\ar[r]\ar[d]^{k_B} & 0\ar[r]\ar[d] & ...\\...\ar[r] & 0\ar[r] & A\ar[r]^f & B\ar[r] & 0\ar[r] & ...}$$ 
It is easy to see that we get a surjective ring homomorphism $\Omega: A\rightarrow End_{\mathcal{K}(A)}(K_f)$, whose kernel can be described by homotopy. It turns out that an element $a$ in $A$ lies in the kernel of $\Omega$ if and only if it exists $h$ in $Hom_A(B,A)$ with $h(1_B)=a$ making the diagram
$$\xymatrix{...\ar[r] & 0\ar[r]\ar[d] & A\ar[r]^f\ar[d] & B\ar[r]\ar[d]\ar[dl]_h & 0\ar[r]\ar[d] & ...\\...\ar[r] & 0\ar[r] & A\ar[r]^f & B\ar[r] & 0\ar[r] & ...}$$ commute. It remains to show that $ker(\Omega)=\tau_B(A)$. It is clear that $ker(\Omega)\subseteq \tau_B(A)$. Conversely, let $a$ be an element in $\tau_B(A)$. Let $h$ be a map in $Hom_A(B,A)$ such that $a=h(b)$ for some $b\in B$. We define a morphism $\tilde{h}\in Hom_A(B,B)\cong End_B(B)$ by mapping $1_B$ to $b$. Therefore, $h\circ\tilde{h}$ lies in $Hom_A(B,A)$ and it satisfies $h\circ\tilde{h}(1_B)=a$. Hence, $a$ lies in $ker(\Omega)$, finishing the proof.
\end{proof}

Following \cite{Wm}, we say that a ring $A$ is derived simple if it does not admit a non-trivial recollement of derived module categories.

\begin{corollary}
If $A$ admits a non-trivial homological 1-finite ring epimorphism  $f: A\rightarrow B$ which is either injective or surjective, then $A$ is not derived simple.
\end{corollary}

Let $f:A\ra B$ be a finite ring epimorphism. It is well-known that, as the trace of a projective $A$-module in $A$, $\tau_B(A)$ is a two-sided idempotent ideal. In particular, if $A$ is a finite dimensional $\mathbb{K}$-algebra, then $\tau_B(A)$ is generated by an idempotent $e$, i.e., $\tau_B(A)=AeA$. More precisely, we have the following easy lemma.

\begin{lemma}\label{idempotent}
If $A$ is a finite dimensional $\mathbb{K}$-algebra, $B$ a finitely generated projective left $A$-module and $I:=\{e_1,...,e_n\}$ a complete set of primitive orthogonal idempotents in $A$, then we have $$\tau_B(A)=\underset{Ae_i|B}{\sum\limits_{e_i\in I}}Ae_iA.$$
\end{lemma}

Following (\cite{CPS}, \S2.1), for a finite dimensional $\mathbb{K}$-algebra $A$, we call an idempotent ideal $AeA$ of $A$ stratifying if the associated ring epimorphism $A\rightarrow A/AeA$ is homological.

\section{Examples}

In this section we will discuss recollements arising from theorem \ref{Main 2} for three classes of homological 1-finite ring epimorphisms. Examples \ref{injective} and \ref{surjective} consider the cases of injective and surjective ring epimorphisms, while proposition \ref{construction finite} and example \ref{otherwise} focus on finite ring epimorphisms which are neither injective nor surjective.

\begin{example}\label{injective}
Let $f:A\rightarrow B$ be a 1-finite, homological and injective ring epimorphism. Then, by corollary \ref{inj surj}, $B$ lies in the same epiclass of $A$ as the universal localisation $A_{\{B/A\}}$ and, by \cite{AS} (theorem 3.5), the finitely generated left $A$-module $T:=A\oplus B/A$ is tilting. Using theorem \ref{Main 2}, we get the following recollement of derived module categories
$$\xymatrix{\mathcal{D}(B)\ar[r]^{}&\mathcal{D}(A)\ar@<1.5ex>[l]_{}\ar@<-1.5ex>[l]_{}\ar[r]^{}&
\mathcal{D}(End_{A}(B/A)).\ar@<1.5ex>_{}[l]\ar@<-1.5ex>_{}[l]}$$
Note that $B/A$ is isomorphic to $K_f$ in $\Dcal(A)$. If $B/A$ is a left $A$-module of projective dimension one, this recollement is precisely the one induced by the universal localisation $A_{\{B/A\}}$ and by the tilting module $T$ in \cite{ALK1} (theorem 4.8).

Indeed, take $A$ to be the quotient of the path algebra over $\mathbb{K}$ of the quiver 
$$
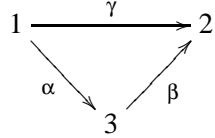
\xymatrix{1\ar[rr]^{\gamma}\ar[dr]_{\alpha} & & 2\\ & 3\ar[ur]_{\beta} &}$$
by the ideal generated by $\beta\alpha$. Consider the map $\gamma^*:P_2\rightarrow P_1$ in $A\mbox{-}proj$ given by multiplication with $\gamma$. 
Using remark \ref{reflection}, it is not difficult to see that $A\rightarrow A_{\{\gamma^*\}}$ is a 1-finite, homological and injective ring epimorphism and, thus, it yields the recollement
$$\xymatrix{\mathcal{D}(A_{\{\gamma^*\}})\ar[r]^{}&\mathcal{D}(A)\ar@<1.5ex>[l]_{}\ar@<-1.5ex>[l]_{}\ar[r]^{}&
\mathcal{D}(End_{\Dcal(A)}(A_{\{\gamma^*\}}/A)).\ar@<1.5ex>_{}[l]\ar@<-1.5ex>_{}[l]}$$
In fact, we can describe explicitly the outer terms of the recollement. On one hand, the universal localisation $A_{\{\gamma^*\}}$ is Morita equivalent to the $\mathbb{K}$-algebra $C$ given by the quotient of the path algebra over $\mathbb{K}$ of the quiver 
$$
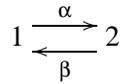
\xymatrix{1\ar[r]<1ex>^{\alpha}& 2\ar[l]<1ex>^{\beta}}$$
by the ideal generated by $\beta\alpha$. On the other hand, since $A_{\{\gamma^*\}}/A$ is isomorphic to $coker(\gamma^*)^{\oplus 2}$ as a left $A$-module, it follows that $End_{\Dcal(A)}(A_{\{\gamma^*\}}/A)$ is isomorphic to $\mathbb{K}\oplus \mathbb{K}$. Moreover, it is easy to check, on a case by case analysis, that this recollement is not induced by a stratifying ideal of $A$.
\end{example}

\begin{example}\label{surjective}
Let $f:A\rightarrow B$ be a 1-finite, homological and surjective ring epimorphism. Then, by corollary \ref{inj surj}, $ker(f)$ is a finitely generated projective left $A$-module and $B\cong A/ker(f)$ lies in the same epiclass of $A$ as the universal localisation $A_{\{ker(f)\}}$. Using theorem \ref{Main 2}, we get the following recollement of derived module categories
$$\xymatrix{\mathcal{D}(A/ker(f))\ar[r]^{}&\mathcal{D}(A)\ar@<1.5ex>[l]_{}\ar@<-1.5ex>[l]_{}\ar[r]^{}&
\mathcal{D}(End_{A}(ker(f))).\ar@<1.5ex>_{}[l]\ar@<-1.5ex>_{}[l]}$$
Note that we have $K_f\cong ker(f)[1]$ in $\Dcal(A)$.

Moreover, if $A$ is a finite dimensional $\mathbb{K}$-algebra then, again by corollary \ref{inj surj}, $B$ and $A/AeA$ lie in the same epiclass of $A$, for some idempotent $e$ in $A$. The above recollement is then the one induced by the stratifying ideal $AeA$ of $A$, namely
$$\xymatrix{\mathcal{D}(A/AeA)\ar[r]^{}&\mathcal{D}(A)\ar@<1.5ex>[l]_{}\ar@<-1.5ex>[l]_{}\ar[r]^{}&
\mathcal{D}(eAe).\ar@<1.5ex>_{}[l]\ar@<-1.5ex>_{}[l]}$$
\end{example}

We now give sufficient conditions for universal localisations to yield finite ring epimorphisms. In what follows, an element $w\not= 0$ of an admissible ideal $I$ of the path algebra of a quiver is called  a relation if it is a linear combination of paths with the same source and target such that for any non-trivial factorisation $w=uv$ neither $u$ nor $v$ lie in $I$. Note that $I$ is generated by its relations.

\begin{proposition}\label{construction finite}
Let $A=\Kbb Q/I$ be a finite dimensional $\mathbb{K}$-algebra given by a connected quiver $Q$ and an admissible ideal $I$ in $\Kbb Q$. Assume that there are vertices $i$ and $j$ and an arrow $\alpha:i\rightarrow j$ in $Q$ such that:
\begin{enumerate}
\item $\alpha$ is the unique arrow in $Q$ starting at vertex $i$;
\item $\alpha$ is the unique arrow in $Q$ ending at vertex $j$;
\item there is no relation in $I$ ending at vertex $j$.
\end{enumerate}
Then the ring epimorphism $f:A\rightarrow A_{\{\alpha^*\}}$ is finite, where $\alpha^*:P_j\rightarrow P_i$ is the map in $A\mbox{-}proj$ given by multiplication with $\alpha$. Moreover, $f_*$ induces a recollement of derived module categories 
$$\xymatrix{\mathcal{D}(A_{\{\alpha^*\}})\ar[r]^{}&\mathcal{D}(A)\ar@<1.5ex>[l]_{}\ar@<-1.5ex>[l]_{}\ar[r]^{}&
\mathcal{D}(\mathbb{K}).\ar@<1.5ex>_{}[l]\ar@<-1.5ex>_{}[l]}$$
\end{proposition}

\begin{proof}
By our combinatorial assumptions and lemma \ref{loc mod}, it is easy to check the following isomorphism of left $A$-modules for each indecomposable projective $A$-module $P_k$ 
$$A_{\{\alpha^*\}}\otimes_A P_k\cong\left\{\begin{array}{cl} P_k, & \mbox{}k\not= j\\ P_i, & \mbox{} k=j.\end{array}\right.$$
Using remark \ref{reflection}, we conclude that $f:A\rightarrow A_{\{\alpha^*\}}$ is a finite ring epimorphism and, when regarded as an $A$-module homomorphism, 
$$f:\bigoplus\limits_k P_k\longrightarrow \bigoplus\limits_k (A_{\{\alpha^*\}}\otimes_A P_k)$$ is given by right multiplication with the square matrix
$${\Tiny{\begin{pmatrix}
1 & & & & & & \\
& ... & & & & & \\
& & 1 & & & & \\
& & & \alpha & & &\\
& & & & 1 & & \\
& & & & & ... & \\
& & & & & & 1
\end{pmatrix}}},$$
where $\alpha$ lies in position $(j,j)$.

We now show that $Hom_A(coker(f),ker(f))=0$. Clearly, we have
$$coker(f)=coker(\alpha^*)=S_i,$$
$$ker(f)=ker(\alpha^*).$$
Note that $f$ is injective if and only if there is no relation in $I$ starting at vertex $i$. Now assume that $Hom_A(coker(f), ker(f))=Hom_A(S_i,ker(\alpha^*))\not= 0$. Consequently, there is a non-trivial element $u$ in $e_iAe_j$ such that $\alpha u$ is zero in $A$, a contradiction to condition (3) in the assumptions. Therefore, by theorem \ref{Main 2}, we get the following recollement of derived module categories
$$\xymatrix{\mathcal{D}(A_{\{\alpha^*\}})\ar[r]^{}&\mathcal{D}(A)\ar@<1.5ex>[l]_{}\ar@<-1.5ex>[l]_{}\ar[r]^{}&
\mathcal{D}(A/\tau_{A_{\{\alpha^*\}}}(A)),\ar@<1.5ex>_{}[l]\ar@<-1.5ex>_{}[l]}$$
where, by lemma \ref{idempotent}, $\tau_{A_{\{\alpha^*\}}}(A)$ is isomorphic to $AeA$ for $e:=\underset{k\not= j}{\sum}e_k.$ Hence, we have 
$$A/\tau_{A_{\{\alpha^*\}}}(A)\cong A/AeA\cong \mathbb{K}.$$
\end{proof} 

\begin{remark}
Note that similar conditions to the ones above are considered in \cite{CK} (example 3.6.2), in the setting of expansions of abelian categories. Indeed, they prove that the inclusion functor $\Xcal_{A_{\{\alpha^*\}}}\hookrightarrow A\mbox{-}mod$ is a right expansion. It is also a left expansion if the map $\alpha^*$ is injective.
\end{remark}

We provide an application for the proposition.

\begin{example}\label{otherwise}
Let $n\in\mathbb{N}_{>1}$ and $A$ be the quotient of the path algebra over $\Kbb$ of the quiver $Q$ below
{\tiny
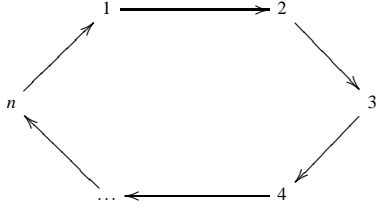
{$$\xymatrix{ & 1\ar[rr] & & 2\ar[dr] & \\ n\ar[ur] & & & & 3\ar[dl]\\ & \dots\ar[ul] & & 4\ar[ll] &}$$}}by an admissible ideal $I$ which is not a power of the ideal generated by the arrows of $Q$. Consequently, there are vertices $i$ and $j$ and an arrow $\alpha:i\ra j$ in $Q$ such that there is no relation in $I$ ending at vertex $j$. We can now apply proposition \ref{construction finite}, yielding the recollement  
$$\xymatrix{\mathcal{D}(A_{\{\alpha^*\}})\ar[r]^{}&\mathcal{D}(A)\ar@<1.5ex>[l]_{}\ar@<-1.5ex>[l]_{}\ar[r]^{}&
\mathcal{D}(\mathbb{K}).\ar@<1.5ex>_{}[l]\ar@<-1.5ex>_{}[l]}$$
In particular, $A$ is not derived simple.
This conclusion can also be obtained by observing that $A$ admits a stratifying ideal $AeA$, for some idempotent $e$ in $A$. Again by assumption, there are vertices $r$ and $s$ and an arrow $\beta:r\rightarrow s$ in $Q$ such that there is no relation in $I$ starting at vertex $r$. Hence, by multiplication with $\beta$ we get an injective morphism $\beta^*:P_s\rightarrow P_r$ and $coker(\beta^*)=S_r$ is of projective dimension 1. Now consider the universal localisation of $A$ at $\Ucal :=\{{\underset{k\not= r}{\bigoplus} P_k}\}$, where $A_{\{\Ucal\}}$ lies in the same epiclass of $A$ as $A/AeA$ for $e:=\underset{k\not= r}{\sum}e_k$. Since $\Xcal_{A_{\{\Ucal\}}}$ is equivalent to $add\{S_r\}$, the ring epimorphism $A\rightarrow A_{\{\Ucal\}}$ is 1-finite and, hence, homological. We conclude that the idempotent ideal $AeA$ is stratifying and it yields the following recollement of derived module categories
$$\xymatrix{\mathcal{D}(\mathbb{K})\ar[r]^{}&\mathcal{D}(A)\ar@<1.5ex>[l]_{}\ar@<-1.5ex>[l]_{}\ar[r]^{}&
\mathcal{D}(eAe).\ar@<1.5ex>_{}[l]\ar@<-1.5ex>_{}[l]}$$

Note that in many cases the algebra $eAe$ in the above recollement can be chosen to be Morita equivalent to $A_{\{\alpha^*\}}$.
For example, let $B$ be the quotient of the path algebra over $\Kbb$ of the quiver 
$$
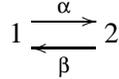
\xymatrix{1\ar[r]<1ex>^{\alpha}& 2\ar[l]<1ex>^{\beta}}$$
by the ideal generated by $\beta\alpha\beta$.
On one hand, the finite ring epimorphism $A\rightarrow A_{\{\alpha^*\}}$, where $A_{\{\alpha^*\}}$ is Morita equivalent to $\mathbb{K}[x]/x^2$, yields the recollement
$$\xymatrix{\mathcal{D}(\mathbb{K}[x]/x^2)\ar[r]^{}&\mathcal{D}(A)\ar@<1.5ex>[l]_{}\ar@<-1.5ex>[l]_{}\ar[r]^{}&
\mathcal{D}(\mathbb{K}).\ar@<1.5ex>_{}[l]\ar@<-1.5ex>_{}[l]}$$
On the other hand, the stratifying ideal $Ae_2A$ induces the recollement
$$\xymatrix{\mathcal{D}(\mathbb{K})\ar[r]^{}&\mathcal{D}(A)\ar@<1.5ex>[l]_{}\ar@<-1.5ex>[l]_{}\ar[r]^{}&
\mathcal{D}(e_2Ae_2),\ar@<1.5ex>_{}[l]\ar@<-1.5ex>_{}[l]}$$
where $e_2Ae_2$ and $\mathbb{K}[x]/x^2$ are isomorphic as rings.
\end{example}

\end{document}